\documentclass[11pt]{amsart}
\usepackage{amsmath}
\usepackage{amsfonts}
\usepackage{amsthm}
\usepackage{amssymb}
\usepackage{amscd}
\usepackage[all]{xy}
\usepackage{enumerate}
\usepackage{color}
\usepackage{bm}
\usepackage{amstext}

\usepackage{graphicx}


\subjclass{METTERE}

\theoremstyle{plain}
\newtheorem{thm}{Theorem}[subsection]

\newtheorem{prop}[thm]{Proposition}

\newtheorem{cor}[thm]{Corollary}
\newtheorem{lem}[thm]{Lemma}

\theoremstyle{definition}
\newtheorem{defn}[thm]{Definition}

\newcommand{\sL}{\mathcal{L}}

\newcommand{\sO}{\mathcal{O}}
\newcommand{\sP}{\mathcal{P}}
\newcommand{\sQ}{\mathcal{Q}}

\newcommand{\sS}{\mathcal{S}}

\newcommand{\sZ}{\mathcal{Z}}

\newcommand{\mC}{\mathbb{C}}

\newcommand{\mP}{\mathbb{P}}

\numberwithin{equation}{section}

\newenvironment{fcaption}{\begin{list}{}{
\setlength{\leftmargin}{35pt}
\setlength{\rightmargin}{35pt}
\setlength{\labelsep}{5pt}
}}{\end{list}}

\author{Pietro Corvaja}
\address{D.M.I.F. \\
the University of Udine\\
Udine, 33100, Italy
\texttt{pietro.corvaja@uniud.it}}

\author{Francesco Zucconi}
\address{D.M.I.F. \\
the University of Udine\\
Udine, 33100, Italy
\texttt{francesco.zucconi@uniud.it}}

\begin{document}

\begin{center}
\textbf{\Large
The surface of  Gauss double points}
\par
\end{center}
{\Large \par}

$\;$

\begin{center}
Pietro Corvaja, Francesco Zucconi 
\par\end{center}

\vspace{5pt}

$\;$

\begin{fcaption} {\small  \item 
Abstract. We study the surface of Gauss double points associated to a very general quartic surface and the natural morphisms associated to it.

}\end{fcaption}

\vspace{0.5cm}


\markboth{Pietro Corvaja, Francesco Zucconi}{}


\section{Introduction}
\subsection{The result}
We work over $\mC$, the complex number field. In this paper $X\subset \mP^3$ is a very general quartic surface. In particular there are no lines inside it. 
We explicitly describe a singular surface $\Sigma_{\rm{dou}}$ naturally associated to $X$. To the best of our knowledge $\Sigma_{\rm{dou}}$ has not been studied yet.
In order to introduce and to study $\Sigma_{\rm{dou}}$ we revise some aspects of the geometry of quartic surfaces.
 We ground our exposition on \cite{Wi}. Indeed \cite{Wi} contains also both a well written modern account about the projective Gauss map and a very nice description of the geometry of a general quartic surface.

Given the surface $X$, it is well known that its bitangents describe a smooth algebraic surface $S$ inside the Gassmannian $\mathbb G(2,4)$ of lines of $\mP^3$. 
If $\sQ_{S}$ is the restriction to $S$ of the universal bundle $\sQ$ over $\mathbb G(2,4)$, it is also well-know that inside its Grothendieck's 
 projectivisation $\mP(\sQ_{S})$ we can define the surface $Y\subset\mP(\sQ)$ of contact points associated to $S$. Indeed a point $y\in Y$ is a couple $y=([l],p)$ where $l$ is a bitangent line of $X$ and $p\in l\cap X$. 
 The geometry of $Y$ is quite well-known; c.f. see: Proposition \ref{ilrivestimento doppio ramificato} and the references there. Together with $S$ and $Y$ there is a third surface which naturally comes with the geometry of $X$. 
 To describe it, we consider a point $p\in X$ and we denote by $T_pX$ the projective tangent space of $X$ at $p$. It is known that the closure of the loci of those points $p\in X$ such that the hyperplane section $X_p:=T_pX\cap X$
  has geometrical genus strictly less than $2$ is a $1$-dimensional subscheme $C_{{\rm{dou}}}$ inside $X$; c.f. see: Theorem \ref{dovesta}.  We will see that for a {\it{general}} point $p\in C_{\rm{dou}}$
 there exists another unique singular point $p'\in X_p$ such that $T_pX=T_{p'}X$. Clearly both $p,p'\in C_{\rm{dou}}$ and since $p\neq p'$ there exists a unique line $l_{p,p'}:=\langle p,p'\rangle\subset\mP^3$ passing throught $p,p'$. 
 Clearly
$[l_{p,p'}]\in S\subset\mathbb G$. The new singular surface we have mentioned above is the locus $\Sigma_{\rm{dou}}\subset\mP^3$ obtained by the closure of the surface swept by the lines $l_{p,p'}$ where $p\in C_{\rm{dou}}$ is a general point. We give a description of the geometry of $\Sigma_{\rm{dou}}$. The projective bundle $\mP(\sQ_S)$ is easily seen to coincide with the following incidence variety: $\{ ([l],p)\in S\times\mP^3\mid p\in l\}$. One of the most basic morphisms associated to the geometry of $X$ is the forgetful one:
$$
f\colon \mP(\sQ_{S})\to \mP^3,\, ([l],p)\mapsto p
$$
Nevertheless it has not been yet deeply studied. We denote by $B(f)\hookrightarrow \mP^3$
 its branch locus and by $R(f)\hookrightarrow\mP(\sQ_S)$ its ramification one. We prove that there exists a curve $C\subset S$ which parameterises the bitangent lines associated to the couples $p, p'\in C_{\rm{dou}}$ where $T_pX=T_{p'}X$. Geometrically $C\subset S$ is birational to the quotient of $C_{\mathrm{dou}}\subset X$ by the natural involution $p\mapsto p'$ where $[l_{p,p'}]\in S$.

 In particular we can construct the ruled surface $\pi_C\colon \Sigma\to C$ obtained by the pull-back of the natural morphism $\pi_S\colon\mP(\sQ_S)\to S$ via the natural inclusion $C\hookrightarrow S$. We can prove:
 

\medskip
\noindent
{\bf{Main Theorem.}} {\it{If $X$ is a very general quartic surface then $S$ and $Y$ are smooth surfaces and $C_{\rm{dou}}$, $C^{\vee}_{\rm{dou}}$ are singular curves whose singularities are fully classified. The morphism $f\colon \mP(\sQ_S)\to\mP^3$ is finite of degree $12$. The almost ruled surface of Gauss double points $\Sigma_{\rm{dou} }$ has degree $160$ and osculates the quartic $X$ along $C_{\rm{dou}}$. It holds:
\begin{enumerate}[{1.}]
\item $B(f)=X\cup \Sigma_{\rm{dou} }$
\item $R(f)=Y\cup \Sigma$.
\end{enumerate}
Moreover the natural morphism $f\colon \mP(\sQ_S)\to\mP^3$ induces by restriction two morphisms $\rho\colon Y\to X$ and $\rho_{\rm{dou}}\colon \Sigma \to  \Sigma_{\rm{dou} }$ of degree respectively $6$ and $1$.}}
\medskip

\noindent
For the proof of the Main Theorem; see: Subsection \ref{subs.teoremageometrico}. Actually we need a detailed description of the geometry associated to the curves $C_{\rm{dou}}$, $C^{\vee}_{\rm{dou}}$. This geometry is shown in Proposition \ref{normalizzodouble}. The proof of  Proposition \ref{normalizzodouble} is postponed in Subsection \ref{subs.proposizionegeometrica}.

Finally we like to mention that this work was originally motivated by Diophantine problems: smooth quartic surfaces are particularly interesting in Diophantine geometry, since they lie at the frontier between rational surfaces, for which the distribution of rational points is well understod, and surfaces of general type, for which it is conjecture that their rational points are never Zariski-dense; for quartic surfaces defined over a number field, it is widely believed that their rational points become  Zariski-dense after a suitable finite extension of their field of definition, but this is proved only in very particular cases, and  no example with Picard number one is known  where this density can be proved. 

Quartic surfaces are limiting cases also for the problem of integral points on open subsets of $\mP^3$: the complement of a surface of degree $\leq 3$ in $\mP^3$ is known to have potential density of integral points; for complements of smooth surfaces of degree five or more, Vojta's conjecture predicts degeneracy (but no case is known for the complement of a smooth surface).  The complement of quartic surface should have a potentially dense set of integral points, but again this is still a widely  open problem. Again by Vojta's conjecture, we expect that removing the union of a quartic surface and any other surface from $\mP^3$ produces an affine variety with degenerate sets of integral points. As a by product of this work and our previous work \cite{CZ}, we could, for instance, deduce the finiteness of the set of integral points, over every ring of $S$-integers, on the complement of the union of a quartic surface $X$ and its associated surface $\Sigma_{\mathrm{dou}}$ to be described in the present work. We intend to devote a future paper to the arithmetic applications of these geometrical investigations.

\section{The morphism of bitangents} 

We agree that a general point on a variety $M$ satisfies a property $\sP$, if there exists an open dense subset of $M$ satisfying the property $\sP$ and that 
 a very general point on $M$ satisfies a property $\sP$ , if there exists a countable union $\sZ$ of closed proper subsets of $X$ such that all the points outside $\sZ$ satisfy the property $\sP$.

In this section we introduce some of the geometrical objects which appear in the statement of the Main Theorem.

Let $V$ be a complex vector space of dimension $4$ and let $V^{\vee}$ its $\mathbb C$-dual. We set $\mP^3:=\mP(V^{\vee})$. Let $F\in {\rm{Sym}}^{4}V$ and let 
$X:=(F=0)\subset \mP^3$ be the associated quartic surface. In this paper, unless otherwise stated, we assume that $X$ is a very general quartic. Indeed we need to use Yau-Zaslov formula; see \cite[Formula: 13.4.2]{Hu}.

The surface $X$ comes naturally with three other surfaces we are going to describe.
\subsection{The surface of bitangents}
Let $\mathbb G:=\mathbb G(2, V^{\vee})$ be the Grassmann variety which parameterises the lines of $\mP^3$. 
\begin{defn}\label{definizionedibitangente} A line $l\subset\mP^3$ is a bitangent line to $X$ if the subscheme $X_{|l}\hookrightarrow l$ is non reduced over each supporting point.
\end{defn}
We denote by $S\subset \mathbb G$ the scheme parameterising bitangent lines of $X$; that is:
\begin{defn} We call
\begin{equation}\label{superficiebitangenti}
S:=\{ [l]\in\mathbb G\mid X_{|l}\, {\rm{is}}\, {\rm{a}}\, {\rm{bitangent}}\, {\rm{to}} X\}.
\end{equation}
{\it{the variety of bitangents to $X$}}.
\end{defn}

\subsection{The surface of contact points}
We have the standard exact sequence of vector bundles on $\mathbb G$:
\begin{equation}
0\to \sQ^{\vee}\to V^{\vee}\otimes\sO_\mathbb G\to \sS\to 0.
\end{equation}
A point $\alpha\in \sQ^{\vee}$ is a couple $([l],p)$ where $[l]\in\mathbb G$, $p\in l\subset \mP^3$. We denote by $\mP(\sQ)$ the variety ${\rm{Proj}}({\rm{Sym}}(Q))$. By definition $\mP(\sQ)$ coincides with the universal family of lines over $\mathbb G$:
$$
\mP(\sQ)=\{ ([l],p) \in \mathbb G \times \mP^3 \mid p\in l \}.
$$
We denote by $\pi_{\mathbb G}\colon \mP(\sQ)\to\mathbb G$ the natural projection and following the mainstream we call {\it{the universal exact sequence}} the following one:
\begin{equation}\label{universal}
0\to \sS^{\vee}\to V\otimes\sO_\mathbb G\to \sQ\to 0.
\end{equation}
By the inclusion $j_{S}\colon S\hookrightarrow \mathbb G$ we can define $\sQ_{S}:=j_{S}^{\star}\sQ$. It remains defined the variety of contact points.
\begin{defn}\label{superficiebitangentipuntate}
We call
\begin{equation}
Y:=\{ ([l],p)\in S\times X\mid p\hookrightarrow X_{|l}  \}.
\end{equation}
{\it{the variety of contact points}}.
\end{defn}
Obviously there is an embedding $j_Y\colon Y\hookrightarrow \mP(\sQ_{S})$ and the natural morphism $\pi_{S}\colon \mP(\sQ_{S})\to S$ restricts to a morphism $\pi\colon Y\to S$ which we call {\it{the forgetful morphism}}.

\subsection{The double cover subscheme }
 We will need to introduce the subscheme of $X$ given by those points $p$ such that the restriction of $X$ to the tangent plane of $X$ at $p$ is a curve $X_p$ of geometrical genus less than $2$. It is a $1$-dimensional subscheme $C_{{\rm{dou}}}$ inside $X$. Following the literature:
\begin{defn}\label{doublecurve}
We call
\begin{equation}
C_{{\rm{dou}}}:=\{ p\in X\mid g(X_p)\leq 1\}.
\end{equation}
{\it{the  double cover subscheme of $X$}}.
\end{defn}
We will see in Proposition \ref{doppioricoprimento} and in Proposition \ref{propietanumeriche} that if $X$ is general then $C_{{\rm{dou}}}$ is an irreducible singular curve.
\subsubsection{The double cover subscheme and the Gauss map}
On $X$ it is defined the Gauss map  $\phi_{{\rm{Gauss}}}\colon X\to  (\mP^3)^{\vee}$, which is a morphism. We set $$C_{{\rm{dou}}}^{\vee}: =\phi_{{\rm{Gauss}}} (C_{\rm{dou}}).$$
Let $\nu_{{\rm{dou}}} \colon {\widetilde{ C_{{\rm{dou}}}}}\to C_{{\rm{dou}}}$ and  respectively $\nu_{{\rm{dou}}}^{\vee}\colon {\widetilde{ C^{\vee}_{{\rm{dou}}}} }\to C_{{\rm{dou}}}^{\vee}$ be the normalisation morphisms. The following Proposition is of interest in itself:
\begin{prop}\label{normalizzodouble} There exist embeddings 
$$j_{{\rm{dou}}} \colon {\widetilde{ C_{{\rm{dou}}}}}\hookrightarrow Y\,\, {\rm{and}}\,\, 
 j^{\vee}_{{\rm{dou}}}\colon {\widetilde{ C^{\vee}_{{\rm{dou}}}} }\hookrightarrow S
 $$
 such that the following diagram is commutative:
\begin{equation}\label{diagrammadellenormalizzodouble}
\xymatrix{
&C_{{\rm{dou}}} \ar[d]^{     \phi_{{\rm{Gauss}}}    }&\ar[l]_{ \nu_{\rm{dou}}   }   {\widetilde{ C_{{\rm{dou}}}}}   \ar[d]^{    \pi_{{\rm{dou}}}    } \ar[r]^{j_{ \rm{dou}}} &Y \ar[d]^{\pi} &\\
& C^{\vee}_{{\rm{dou}}} &    \ar[l]_{ \nu^{\vee}_{\rm{dou}}   }    {\widetilde{ C^{\vee}_{{\rm{dou}}}} } \ar[r]^{j^{\vee}_{ \rm{dou}}} & S&
}
\end{equation}
\noindent
where $\pi_{{\rm{dou}}}\colon  {\widetilde{ C_{{\rm{dou}}}}}\to {\widetilde{ C^{\vee}_{{\rm{dou}}}}}$ is a $2$-to-$1$ branched covering induced by the restriction of $\pi\colon Y\to S$ to $ {\widetilde{ C_{{\rm{dou}}}}}$.
\end{prop}
\noindent
For the proof of the above Proposition \ref{normalizzodouble}; see: Subsection \ref{subs.proposizionegeometrica}.

\subsection{The almost ruled surface of  Gauss double points}
We will need to consider also a third surface naturally associated to $X$. We have found no reference on it. We will see that for a {\it{general}} point $p\in C_{\rm{dou}}$ there exists another unique singular point $p'\in X_p$ such that $T_pX=T_{p'}X$. Clearly both $p,p'\in C_{\rm{dou}}$ and since $p\neq p'$ there exists a unique line $l_{p,p'}:=\langle p,p'\rangle\subset\mP^3$ passing throught $p,p'$. We will show that 
$[l_{p,p'}]\in  {j^{\vee}_{ \rm{dou}}}{\widetilde{ C^{\vee}_{{\rm{dou}}}} } \subset S\subset\mathbb G$. On the other hand the natural morphism $\rho_{\mP^3}\colon\mP(\sQ)\to\mP^3$ if restricted to $\mP(\sQ_{S})$  gives a morphism $f\colon \mP(\sQ_{S})\to\mP^3$. We will show that the image 
$\Sigma_{\rm{dou} }:=f(\pi_S^{-1}( {j^{\vee}_{ \rm{dou}}}{\widetilde{ C^{\vee}_{{\rm{dou}}}} }   )$ is a surface inside $\mP^3$ which is the closure of the surface swept by the lines $l_{p,p'}$ as $p\in C_{\rm{dou}}$ is a general point.
\begin{defn} \label{rigariga} We call the subscheme $\Sigma_{\rm{dou} }\hookrightarrow\mP^3$ given by the closure of the surface swept by the lines $l_{p,p'}$ where $p\in C_{\rm{dou}}$ is a general point {\it{the almost ruled surface of  Gauss double points}}.
\end{defn}

\subsection{The basic morphism} Note that geometrically $\mP(\sQ_{S})$ is easily seen as
\begin{equation}\label{lavarietàdellerettepuntate}
 \mP(\sQ_{S})=\{ ([l],p)\in S\times\mP^3\mid p\in l\}.
\end{equation}
We have introduced above a basic object of this geometry: the morphism $f\colon \mP(\sQ_{S})\to\mP^3$. It is obtained by projecting to the second factor and it is called the morphism of bitangents. 
We want to understand its branch locus subscheme $B(f)\hookrightarrow\mP^3$ and its ramification one $R(f)\hookrightarrow \mP(\sQ_S)$. 

\section{Special curves on a quartic surface} 

\subsection{Singularities of a plane quartic}
We recall a basic fact on irreducible plane quartic.
\begin{lem}\label{singularitiesofplanequartics} Let $C\subset\mP^2$ be an irreducible plane quartic. Then $C$ has at most three singularities. If $g(C)=1$ then the following cases occurs
\begin{enumerate}[{(1)}]
\item a  point of multiplicity $3$, or
\item a tacnode, or
\item two nodes, or
\item a node and a cusp, or
\item two cusps.
\end{enumerate}
If $g(C)=2$ then $C$ has exactly one node or one cusp.
\end{lem}
\begin{proof} Easy.
\end{proof}

\subsection{Classification of points of a general quartic surface} 
Let $F\in \mathbb C[x_0,x_1,x_2,x_3]$ be as above a general homogeneous polynomial of degree $4$ and let $X:=V(F)\subset \mP^3$ be the corresponding quartic surface. Since $X$ is smooth we can define the   Gauss map, which maps a point $X\ni p$ to its tangent space seen as a point of the dual projective space: 
$$
\phi_{ {\rm{Gauss}}}\colon X\ni p\to {{T_{p}X}}\in (\mP^3)^{\vee} 
$$

Clearly $\phi_{{\rm{Gauss}}}\colon X\to  (\mP^3)^{\vee}$ is a morphism since $X$ is smooth. In the rest of this section we strongly rely on \cite{Wi}. Actually, by the stability property of the Gauss map, c.f. \cite[Section 2.1]{Wi}, we know the analytic behaviour of the Gauss map in an analytic neighborhood of any point $p\in X$. It follows that for any point $p\in X$ it holds that $d\phi_{ {\rm{Gauss}},\, p}\neq 0$; see \cite[Proposition 2.15]{Wi}. 

\begin{prop}\label{tangentcurves} Let $X$ be a general quartic surface. Then:
\begin{enumerate} [{(1)}]
\item All tangent curves are irreducible
\item Tangent curves have only singularities of multiplicity $2$ (and the second fundamental form is non zero for any $p\in X$).
\item All rational tangent curves are nodal.
\item There are no elliptic cuspidal tangent curves on  $X$
\end{enumerate}
\end{prop}
\begin{proof}  $(1)$ follows since ${\rm{Pic}}(X)=\mathbb Z$. $(2)$ see c.f. \cite[Lemma 2.1.4]{Wi}. $(3)$ see c.f. \cite[Fact 1.4.8 p. 21]{Wi}. $(4)$ see \cite[Lemma 2.1.6]{Wi}.
\end{proof}
It is interesting to stress here that Proposition \ref{tangentcurves} $(4)$ means that if $X$ is a general quartic then there are no elliptic hyperplane sections with two cusps. This should be read together with Lemma \ref{singularitiesofplanequartics}  $(5)$. This leads to a full classification of the possible hyperplane sections
 $$
 X_p=X\cap T_pX,\,\, p\in X.$$
\begin{prop}\label{classification of tangentsurves} Let $X$ be a very general quartic $X\subset\mP^3$. Let $p\in X$ and $X_p=T_pX\cap X$. The couple $(X_p,p)$ is one of the following types:

\begin{enumerate}[{(1)}]
\item general case: $g(X_p)=2$ and $X_p$ has only one node.\\
 \begin{center}\includegraphics[width=.2\linewidth, height=2cm]{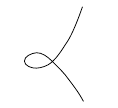}\end{center}
\bigskip

\item simple parabolic point: $g(X_p)=2$ and $X_p$ has only one cusp\\
\begin{center}\includegraphics[width=.2\linewidth, height=1.5cm]{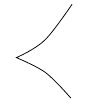}\end{center}
\bigskip

\item simple Gauss double point: $g(X_p)=1$ and $X_p$ has only two nodes\\
\begin{center}\includegraphics[width=.16\linewidth, height=1.5cm]{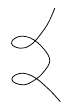}\end{center}
\bigskip

\item parabolic Gauss double point:  $g(X_p)=1$ and $X_p$ has a cusp on $p$ and a node on another point $p'\neq p$.\\\begin{center}\includegraphics[width=.15\linewidth, height=1.5cm]{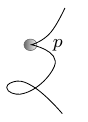}
\end{center}
\bigskip

\item dual to parabolic Gauss double point:  $g(X_p)=1$ and $X_p$ has a cusp on $p'\neq p$ and a node on $p$.\\
\begin{center}
\includegraphics[width=.14\linewidth, height=1.5cm]{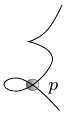}
\end{center}
\bigskip

\item Gauss swallowtail: $g(X_p)=1$ and $X_p$ has one  tacnode.\\
\begin{center}
\includegraphics[width=.18\linewidth, height=1.5cm]{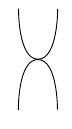}
\end{center}
\bigskip

\item Gauss triple point: $g(X_p)=0$ and $X_p$ has three distinct nodes\\
\begin{center}
\includegraphics[width=.18\linewidth, height=1.5cm]{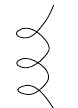}
\end{center}
\end{enumerate}
\end{prop}
\begin{proof} The reader should have a glance to the pictures of \cite[Proposition 2.1.7]{Wi}. 
\end{proof}

By Proposition \ref{classification of tangentsurves} and by the calculation of the Milnor numbers of nodes, cusps, and tacnodes it follows the full classification of the 
singular points of the image $S^{\vee}:=\phi_{\rm{Gauss}}(S)$.

\begin{prop}\label{doublepointsgaussimage} Let $p^{\vee}:=\phi_{\rm{Gauss}}(p)$. Then ${\rm{mult}}_{p^{\vee}}S^{\vee}=2$ if $p$ is a simple parabolic point or a simple Gauss double point and ${\rm{mult}}_{p^{\vee}}S^{\vee}=3$ if $p$ is a parabolic Gauss double point, dual to parabolic Gauss double point, a Gauss swallowtail, or a Gauss triple point.
\end{prop}
\begin{proof} See \cite[Proposition 2.19]{Wi}.
\end{proof}

\subsection{The parabolic curve of a general quartic surface}
There are three curves on $X$ which contain  important information on $X$. The first one is the following:

\begin{defn} We define the parabolic curve $C_{{\rm{par}}}\subset X$ to be the ramification locus of the Gauss map $\phi\colon X\to X^*$. A point $p\in X$ is called parabolic if $p\in C_{{\rm{par}}}$.
\end{defn}

\subsubsection{The asymptotic directions} We recall that locally  we can consider a neighbourhood of the point $p=(0,0,0)$ with coordinates $(x,y,z)\in \mathbb C^3$ such that locally $T_pS$ is given by $(z=0)$. 
Hence the germ of $S$ at $p$ is given by $z+q(x,y)\in \mathbb C[[x,y]]$. In particular the local analytic expression of the Hessian of $F$ is
\begin{equation}\label{Hessiano}
{\rm{Hess}}_{(S,p)}:=\begin{vmatrix} \frac{\partial^2 q}{\partial^2_x} & \frac{\partial^2 q}{\partial_x\partial_y}  \\
 \frac{\partial^2 q}{\partial_y\partial_x}& \frac{\partial^2 q}{\partial^2_y}  \end{vmatrix} (p)
\end{equation}
Then each principal direction, that is those giving the tangents to the branches of $X_p$ at $p$, is obtainable by the vector $v\in T_pS$ such that for the induced quadratic form it holds ${\rm{Hess}}_{(S,p)}(v,v)=0$. Following a notation coming from differential geometry these two directions are called {\it{asymptotic directions}} at $p$. In particular if $X_p$ has a cusp on $p$ then there exists a unique direction $v$ such that ${\rm{Hess}}_{(S,p)}(v,...)\equiv 0$ and the line $\langle v\rangle\subset T_pS$ is the direction of ramification of $\phi_{\rm{Gauss}}$ that is $d_{p}\phi_{\rm{Gauss}}(v)=0$.
\medskip

\subsubsection{Classification of the parabolic points}
By a local analysis it follows that 
\begin{prop} \label{parab} Let $X$ be a general quartic. A point $p\in X$ is parabolic iff $p$ is a cusp or a tacnode of $X_p$. The parabolic curve $C_{{\rm{par}}}$ is the zero locus of the determinant of the Hessian of $X$. 
Moreover $C_{{\rm{par}}}$ is a smooth element of the linear system $ |8h|$. In particular $C_{{\rm{par}}}$ has genus $129$.
\end{prop}
\begin{proof} We have noted above that $C_{\rm{par}}$ is the locus where the Gauss morphism is not a smooth one. Hence by the local analytic description of the Gauss map the claim follows. See c.f. \cite[Proposition 2.2.4]{Wi}.
\end{proof}

\subsection{The flecnodal curve}
The second curve inside $X$ which is useful to understand the geometry of $X$ is the one containing all the hyperflexes.
\begin{defn} A line $l\subset \mP^3$ is called a hyperflex line if the subscheme $X_{|l}\hookrightarrow l$ is supported over a unique point $p\in X\cap l$. In this case the line $l$ is called an hyperflex line of $X$ at $p$.
\end{defn}
\begin{defn} A point $p\in X$ is called a hyperflex if there exists a hyperflex line through it. We define the hyperflex curve $C_{{\rm{hf}}} \subset X$ to be the reduced scheme of hyperflexes.
\end{defn}
We sum up the result on $C_{{\rm{hf}}}$ we need:
\begin{prop} \label{iperflessssi}Let $X$ be a general quartic. Then it holds:
\begin{enumerate}[{(1)}]
\item   $C_{{\rm{hf}}}$ is irreducible;
\item if $p \in C_{{\rm{hf}}}$ is a general point then $X$ has exactly one hyperflex in $p$;
\item  $C_{{\rm{hf}}}$ has geometric genus $201$;
\item if $p\in  C_{{\rm{hf}}}$ has two distinct hyperflexes then $p$ is a singular point of $C_{{\rm{hf}}}$;
\item  $C_{{\rm{hf}}}\in |20h|$.
\end{enumerate}
\end{prop}
\begin{proof} See c.f. \cite[Subsection 2.3.3]{Wi} and c.f. \cite[Corollary 2.4.6]{Wi}.
\end{proof}

\subsubsection{Swallowtail points, parabolic curve and the flecnodal curve}
On the parabolic curve $C_{{\rm{par}}}$ there is unicity of asymptotic directions. This enables us to construct the subbundle $\sL\subset TX_{|C_{{\rm{par}}}}$ of asymptotic directions. By Proposition \ref{parab} we can build the standard tangent sequence $$0\to T_{C_{{\rm{par}}}}\to  TX_{|C_{{\rm{par}}}} \stackrel{\nu_{C_{{\rm{par}}}}}{\longrightarrow} N_{C_{{\rm{par}}}|X }\to 0$$ It is well-known that the sheme where  $\nu_{C_{{\rm{par}}}|\sL}$ vanishes is the one given by Gauss swallowtails; see c.f. \cite[Remark 2.2.7]{Wi}.

\begin{prop}\label{codadirondine} Every swallowtail is a simple zero of $\nu_{C_{{\rm{par}}}|\sL}$. A point $p\in C_{{\rm{par}}}$ is a Gauss swallowtail iff $p\in C_{{\rm{hf}}}$, that is set theoretically
$$
{\rm{Swallowtail}}(X)= C_{{\rm{par}}}\cap C_{\rm{hf}}
$$
 Moreover every $p\in C_{{\rm{par}}}\cap  C_{{\rm{hf}}}$ is a smooth point both of  $C_{{\rm{hf}}}$ than $C_{\rm{par}}$ and
$${\rm{mult}}_p (C_{{\rm{par}}}\cap  C_{{\rm{hf}}})=2.$$
\end{prop}
\begin{proof} See c.f. \cite[Section 2.4 p. 39]{Wi}.
\end{proof}

\begin{cor}\label{sontrecentoventi} There are 320 Gauss swallowtails.
\end{cor}
\begin{proof} See c.f. \cite[Proposition 2.4.5]{Wi}.
\end{proof}

\subsection{The double cover curve}
By the classification of points of a general quartic surface it is natural to consider the closure of the locus of simple Gauss double points. This gives the third curve on $X$ whose importance to understand the geometry of $X$ has been recognised by many authors; see c.f. \cite{W} and the bibliography of \cite{Wi}.  We have introduced in Definition \ref{doublecurve} the double cover curve $C_{\rm{dou}}\subset X$ as the subset of points $x\in X$ such that $g(X_p)\leq 1$. By Theorem \ref{classification of tangentsurves}, as a set  $C_{\rm{dou}}$ consists exactly of simple Gauss points, parabolic Gauss double points, dual to parabolic Gauss double points, Gauss swallowtails and Gauss triple points.

\begin{prop}\label{doppioricoprimento} The subset $C_{\rm{dou}}\subset X$ is a closed irreducible subscheme of pure dimension $1$.
\end{prop}
\begin{proof} See c.f. \cite[Proposition 2.5.6]{Wi}.
\end{proof}

Note that there exists a rational involution $$j_{\rm{dou}}\colon C_{\rm{dou}}\dashrightarrow C_{\rm{dou}}$$ which sends the node $p\in X_p\cap C_{\rm{dou}}$ to the other node $ p'\in X_p\cap C_{\rm{dou}}$ if $p$ is a general point of $C_{\rm{dou}}$. The computation of the degree of $C_{\rm{dou}}$ requires a certain amount of work on its image $C_{\rm{dou}}^{\vee}$ through the Gauss morphism. Using the apolarity theory we can easily see that letting ${\rm{Pol}}_p(X)$ the polar cubic surface to $X$ with respect to the point $p$, we can define the curve
$$
D_p:=\{q\in X\mid p\in T_qX\}
$$
By construction $D_p=\phi_{\rm{Gauss}}^{-1}(p^{\perp}\cap X^{\vee})$. In other words we can interpret geometrically the Gauss morphism as induced by the morphism $\mP^3\to(\mP^3)^{\vee}$ given by the sublinear system of the polar cubics. This gives

\begin{thm}\label{dovesta} The curve $C_{\rm{dou}}$ belongs to $|80h|$ while $C_{\rm{dou}}^{\vee}$ has degree $480$.
\end{thm}
\begin{proof} See c.f. \cite[Proposition 2.5.20, Corollary 2.5.21]{Wi}.
\end{proof}

The curve $C_{\rm{dou}}$ is singular. We sketch the local analysis necessary to understand the local geometry of $C_{\rm{dou}}$, but we stress that it requires the deep Yau-Zaslov formula which says that there are exactly $3200$ nodal rational curves inside the linear system $|\sO_X(1)|$ if $X$ is general. By Proposition \ref{classification of tangentsurves} $(7)$ we know that each one of these rational nodal curves is a tangent section with three nodes. Each node determines the tangent section and since there are only nodes then there are exactly $3$ nodes. By our previous notation this means that these 3200 nodal rational curves are exactly the hyperplane sections with Gauss triple points. This implies that there are exactly 9600 Gauss triple points. By the local analysis which uses the local stability of the Gauss morphism we have that $C_{\rm{dou}}$ and $C_{\rm{par}}$ intersect only at the 320 Gauss swallowtails, see: Corollary \ref{sontrecentoventi}, with multiplicity two and with multiplicity one at the parabolic Gauss double point, see: \cite[Proposition 2.5.15]{Wi}. Hence we obtain that there are 1920 parabolic Gauss double points and so there are also 1920 dual to a parabolic Gauss double point. The above analysis leads to the following:

\begin{prop}\label{propietanumeriche} The double cover curve $C_{\rm{dou}}\subset X$ is irreducible and it has only ordinary singularities. More precisely 
\begin{enumerate}
\item $C_{\rm{dou}}$ has a node at each point of any Gauss triple;
\item $C_{\rm{dou}}$ has a cusp at each  dual to a parabolic Gauss double point; 
\item Gauss swallowtails are smooth points of $C_{\rm{dou}}\subset X$;
\item $C_{\rm{dou}}$ is smooth at any parabolic Gauss double point.
\end{enumerate}
Finally there are precisely $9.600$ Gauss triple points and 1920 parabolic Gauss double points. In particular the genus of $C_{\rm{dou}}$ is $1281.$
\end{prop}
\begin{proof}It follows by the detailed analysis on the geometry of $C_{\rm{dou}}$ done in \cite [Section 2.5]{Wi}. 
\end{proof}

It also holds:
\begin{prop}\label{dualedouble} The dual curve $C^{\vee}_{\rm{dou}}$ has a three branched node in correspondence to the triplet of Gauss triple points and a cusp at each one of the point corresponding to parabolic Gauss double points. Hence it has genus 561.
\end{prop}
\begin{proof} See \cite[Section 2.7]{Wi}.
\end{proof}

\section{The almost ruled surface}

\subsection{Numerical invariants associated to the surfaces of bitangents} We need to recall briefly some results mainly taken from \cite{T} and \cite{W}. We stress that we have set $\mP^3=\mP(V^\vee)$ and that $X\subset \mP^3=\mP(V^\vee)$ is a very general quartic surface. In particular there are no lines contained inside $X$. The next Proposition is well-known, possibly since very long time ago, but we include a proof of it because in the sequel we need analogue techniques and notation to write the ramification divisor $R(f)$.
\begin{prop}\label{lisciezza} The scheme $S\subset\mathbb G$ which parameterises bitangents to a smooth quartic surface $X\subset \mP^3$ with no lines is a smooth surface.
\end{prop}
\begin{proof} We fix a line $l\subset \mP^3$ which is bitangent to $X$. W.l.o.g. we can assume that $l:=(x_2=x_3=0)$ and that the two points in $X\cap l$ are $P=(1:0:0:0)$, $P_{\lambda}=(1:\lambda:0:0)$ where we do not assume $\lambda\neq 0$. Then 
$$
F(x_0:x_1:x_2:x_3)=x_{1}^2(x_1-\lambda x_0)^2 +x_2G(x_0:x_1:x_2:x_3)+x_3H(x_0:x_1:x_2:x_3)
$$
where $G,H\in \mathbb C[x_0,x_1,x_2,x_3]$ are homogeneous forms of degree $3$. We recall that by generality $l\not\subset X$.

We consider an open neighbourhood $U'\subset \mathbb G$ of $[l]$ and let $(u_0,u_1,u_2,u_3)$ be a regular parameterisation of $U'$ of $[l]$ inside $\mathbb G$; this means that for points $[r]$ close to $[l]$ inside $U'$ we can write
$$
r:=\{ (x_0:x_1: x_0u_0+x_1u_1:x_0u_2+x_1u_3) \mid (x_0:x_1)\in\mP^1\}\subset\mP^3.
$$
We look for conditions on the tangent vector $v:=(u_0 ,u_1,u_2, u_3)\in T_{[l]}\mathbb G$ to be inside the Zariski tangent space $(m_{S,[l]}/m_{S,[l]}^2)^{\vee}$ of $S$ at $[l]$. This means that if in $\mathbb C[x_0, x_1, u_0,u_1,u_2,u_3,\epsilon]$, where $\epsilon^2=0$, we write $$f(x_0:x_1; u_0,u_1,u_2,u_3,\epsilon)=F(x_0:x_1: \epsilon (x_0u_0+x_1u_1):\epsilon(x_0u_2+x_1u_3))$$ then it must exist a polynomial $q\in \mathbb C[x_0, x_1, u_0,u_1,u_2,u_3,\epsilon]$  of degree at most $2$ in the variables $x_0, x_1$ with $f=q^2$. Since $f(x_0 :x_1; u_0, u_1,u_2, u_3, \epsilon)=x_{1}^{2}(x_{1}-\lambda x_{0})^2 + \epsilon ( (x_0u_0+x_1 u_1)g(x_0:x_1) +( x_0u_2+x_1u_3) h(x_0:x_1))$ where $g(x_0:x_1):= G(x_0:x_1:0:0)$ and $h(x_0:x_1):= H(x_0:x_1:0:0)$ this is possible iff $x_{1} (x_{1}-\lambda x_{0})$ is a factor of  $(x_0u_0+x_1u_1)g(x_0:x_1)+(x_0u_2+x_1u_3)h(x_0:x_1)$, . We distinguish now two cases: $\lambda=0$ or $\lambda\neq 0$. If  $\lambda\neq 0$ then we obtain that $v\in (m_{S,[l]}/m_{S,[l]}^2)^{\vee}$ iff

\[
\begin{cases}
u_0g(1:0)+u_2h(1:0)=0 \\
(u_0+\lambda u_1)g(1:\lambda)+(u_2+\lambda u_3)h(1:\lambda)=0
\end{cases}
\]
The above linear system has rank $\leq 1$ iff $P$ or $P_{\lambda}$ is a singular point of $X$. If $\lambda=0$ the condition is equivalent to

\[
\begin{cases}
u_0g(1:0)+u_2h(1:0)=0 \\
u_0\frac{\partial}{\partial x_1}g(1:0)+u_1g(1:0)+u_2\frac{\partial}{\partial x_1}h(1:0)+u_3h(1:0)=0
\end{cases}
\]
and again the rank is less or equal to $1$ iff $P$ is a singular point.
\end{proof}

\subsubsection{Basic diagrams} 
We need a description of $S\subset\mathbb G$ and of its invariants. Denote by $H_{\mathbb G}$ the hyperplane section of the Pl\"ucker embedding $\mathbb G\hookrightarrow \mP(\bigwedge^2V^{\vee})$. We have the universal exact sequence, its restriction over $S$ and we stress that $\sO_\mathbb G(H_{\mathbb G})={\rm{det}}\sQ={\rm{det}}\sS$.

\noindent

Now we consider the standard conormal sequence of $X$ inside $\mP^3$:
\begin{equation}\label{conormal}
0\to \sO_X(-4)\to \Omega^1_{\mP^3|X}\to \Omega^1_X\to 0.
\end{equation}
\noindent
We partially maintain the notation of \cite{T} to help the reader to check some of our assertions. Following \cite{T} we can build the following diagram:

\begin{equation}\label{diagrammasullasuperficie}
\xymatrix{
&J_X\colon \mP(\Omega^1_X(1))\ar[r]\ar[dr]^{\rho_X}&\mP(\Omega^1_{\mP^3|X}(1))\ar[d]^{\rho'}\leftarrow \mP(\sQ_S) \ar[dr]^{\pi_S}&\\
&&X\subset\mP(V^\vee)&S&
}
\end{equation}
\noindent 
where the inclusion $J_X\colon \mP(\Omega^1_X(1))\to \mP(\Omega^1_{\mP^3|X}(1))$ is given by the sequence (\ref{conormal}) and the morphism 
$\mP(\sQ_S)\to \mP(\Omega^1_{\mP^3|X}(1))$ is the restriction over $S$ of the standard diagram:

\begin{equation}\label{diagrammasullospazio}
\xymatrix{
&&\mP(\Omega^1_{\mP^3}(1))\cong\ar[d]^{\rho}\mP(\sQ) \ar[dr]^{\pi_{\mathbb G}}&\\
&&\mP(V^\vee)&\, \,  \mathbb G&
}
\end{equation}
and $\rho'\colon \mP(\Omega^1_{\mP^3|X}(1))\to \mP(V^\vee)$ is the obvious restriction.

\subsubsection{Geometrical interpretation}
By construction the $\mP^1$-bundle $\pi_{\mathbb G}\colon \mP(\sQ)\to\mathbb G$ is the universal family of $\mathbb G$, and the $\mP^2$-bundle $\rho\colon \mP(\Omega^1_{\mP^3}(1))\to\mP^3$ is the projective bundle of the tangent directions on $\mP^3$; that is: $\rho^{-1}(p)=\mP(T_{\mP^3,p})$ where $T_{\mP^3,p}$ is the vector space given by the tangent space to $\mP^3$ at the point $p$. The isomorphism $\mP(\Omega^1_{\mP^3}(1))\cong \mP(\sQ)$ is well-known.

We denote by $N$ the divisor on $\mP(\Omega^1_{\mP^3}(1))$ and by $R$ the divisor on $\mathbb P(\sQ)$ such that 
$$
\rho_\star\sO_{\mP(\Omega^1_{\mP^3}(1))}(N)=\Omega^1_{\mP^3}(1),\, \pi_{\mathbb G\star} (\sO_{\mathbb G}(R)=\sQ
$$
Since no confusion can arise we denote by $R$ also the restriction to $\mP(\sQ_S)$ of $R$, hence $\pi_{S\star}\sO_{\mP(\sQ_S)}(R)=\sQ_S$
We also denote by $T$ the divisor on $\mP(\Omega^1_X(1))=\mP(\Omega^1_X)$ such that:

\begin{equation}\label{notazionesemplice}
\rho_{X\star}(\sO_{\mP(\Omega^1_X)}(T))=\Omega^1_X.
\end{equation}

\begin{lem}\label{projcotangent} For the $3$-fold $\mP(\Omega^1_X(1))$ it holds:
$$\mP(\Omega^1_X(1))\cong \{(p,[l])\in X\times \mathbb G\mid l\in \mP( T_{p} X) \}$$
\end{lem}
\begin{proof} Trivial since $X$ is smooth.
\end{proof}

\subsubsection{Useful divisor classes}
We can relate classes which are easily seen by the geometry of $\rho_X\colon \mP(\Omega^1_X(1))\to X$ to ones which can be better seen via the morphism $\pi_S\circ J_X\colon  \mP(\Omega^1_X(1))\to S$. We think it does not create any confusion to write $R$ for the class 
$J_X^{\star}(R_{\mid \mP(\Omega ^1_{\mP^3|X})}  )   \in {\rm{Pic}}(\mP(\Omega^1_X(1))$ which comes from the divisor $R$ on $\mP(\sQ_S)$ via the inclusion 
$\mP(\sQ_S)\hookrightarrow \mP(\Omega^1_{\mP^3|X}(1))$. Finally we set ${\rm{Pic}}(\mP(\Omega^1_X(1))\ni H_X:= (\pi_S\circ J_X)^{\star}(H_{\mathbb G|S})$ and 
$h:=H_{\mP^3\mid X}\in {\rm{Pic}}(X)$.

\begin{lem}\label{relativeonX}  It holds on ${\rm{Pic}}(\mP(\Omega^1_X(1))):$
\begin{enumerate} 
\item $R\sim \rho_{X}^{\star}h$;
\item $N_{|\mP(\Omega^1_X(1))}\sim T +\rho_{X}^{\star}h$;
\item $H_X\sim T+2\rho_{X}^{\star}h$.
\end{enumerate}
\end{lem}
\begin{proof} Easy.
\end{proof}

\subsubsection{The surface of contact points as a double cover}

Inside $S$ there is the subscheme $B_{\rm{hf}}\hookrightarrow S$ which parameterises the hyperflex lines. In Proposition \ref{iperflessssi} we have recalled a description of the corresponding curve $C_{\rm{hf}}\subset X$. Using the surface of contact points $Y$; see Definition\ref{superficiebitangentipuntate}, and $B_{\rm{hf}}$ we have an important information on $S$:
\begin{prop}\label{ilrivestimento doppio ramificato}
There exists a non trivial $2$-torsion element $\sigma\in {\rm{Div}}(S)$ such that the surface of contact points $Y$ can be realised as a subscheme of $\mP(\sO_S\oplus\sO_S(\sigma+H_{\mathbb G|S}))$. The restriction of the natural projection $\mP(\sO_S\oplus\sO_S(\sigma+H_{|S}))\to S$ induces a $2$-to-$1$ cover $\pi\colon Y\to S$ branched over $B_{\rm{hf}}\in |2H_{\mathbb G|S}|$. In particular $Y$ is a smooth surface. Moreover as a class inside 
$\mP(\sQ_S)$ we have that $Y\in |2R+\pi_S^{\star}(\sigma)|$.
\end{prop}
\begin{proof} See \cite[Proposition 3.11]{W}. See also \cite{CZ}.\end{proof}

By the diagram (\ref{diagrammasullasuperficie}) and by a slight abuse of notation we obtain the basic diagram:

\begin{equation}
\label{oneuniversal&projection}
\xymatrix{ & Y \ar[dl]_{{ \rho}}
\ar[dr]^{\pi}\\
 X\subset\mP^{3} &  & S\subset\mathbb G }
\end{equation}
The fact that $Y$ is a divisor both in $\mP(\Omega^1_X(1))$ than in $\mP(\sQ_S)$ makes possible to link the geometry of $X$ to the geometry of $S$ via the one of $Y$. In particular it is noteworthy that we can try to obtain special curve on $S$ via special curves on $X$ and viceversa.
First we recall that since $X$ is general ${\rm{Pic}}(X)=[h]\mathbb Z$, where we recall that $h:=H_{\mP^3_{|X}}$. Hence 
\begin{equation}\label{picdi}
{\rm{Pic}}(\mP(\Omega^1_X(1)))=[T]\mathbb Z\oplus [R]\mathbb Z.
\end{equation}

\begin{prop}\label{sopraX}
As a class inside ${\rm{Pic}}(\mP(\Omega^1_X(1)))$ it holds that $$Y\in |6T+8R|.$$
\end{prop}
\begin{proof} See \cite[Proposition 2.3]{T}.
\end{proof}
\subsubsection{The class of $S$ in the Chow ring of $\mathbb G$}
It is quite natural to introduce the following classes inside the Chow ring $\oplus_{i=1}^{4}{\rm{CH}}^{i}(\mathbb G)$ associated to the fundamental ladder, point, line, plane, $p\in l\subset h\subset \mP^3$:
${\rm{CH}}^1(\mathbb G)\ni\sigma_{l}:=\{[m]\in\mathbb G| m\cap l\neq \emptyset \}$, 
${\rm{CH}}^2(\mathbb G)\ni\sigma_{p}:=\{[l]\in\mathbb G| p\in l\}$, ${\rm{CH}}^2(\mathbb G)\ni \sigma_{h}:=\{[l]\in \mathbb G| l\subset h\}$. It is well known that 
$$\sigma_l^2=\sigma_h+\sigma_p$$
and that the divisorial class of $\sigma_l$ is the class $H_{\mathbb G}$.
\begin{lem}\label{chowclass} The following identity holds in the Chow ring of $\mathbb G$:
$$
{\rm{CH}}^2(\mathbb G)\ni[S]=40\sigma_l^2+28\sigma_h+12\sigma_p.
$$
In particular  $\rm{deg}(S)=H_{\mathbb G}^2\cdot [S]=40.$
\end{lem} 
\begin{proof}
See \cite[Lemma 3.30]{W}.
\end{proof}

\subsubsection{Numerical invariants }

\begin{thm}\label{formulae} For the surfaces $S$, $Y$ we have the formulae:
\begin{enumerate}
\item  $K_{S}=3H_{\mathbb G|S}+\sigma$, $q(S)=0$, $p_g(S)=45$, $h^1(S_X,\Omega^1_{S}) =100$, $c_2(S)=192$
\item $K_Y=\pi^{\star}(4H_{\mathbb G|S})$, $q(Y)=0$, $p_g(Y)=171$
\end{enumerate}
\end{thm}
\begin{proof} See \cite[Cohomological study pp. 41-45]{W}.
\end{proof}
\subsection{The dual geometry}

We consider the following open set of $C_{\rm{dou}}$:
$$
C_{\rm{dou}}^{0}:=\{p \in C_{\rm{dou}}\mid p\, {\rm{is\, a\, simple\, Gauss\, double\, point}}\}\setminus (C_{\rm{hf}}\cap C_{\rm{dou}})
$$
We set $(C_{\rm{dou}}^{0})^{\vee} :=\phi_{\rm{Gauss}}(C_{\rm{dou}}^{0})$.
\begin{lem} \label{gaussristrettino} The restriction of the Gauss morphism to $C_{\rm{dou}}^{0}$ induces a $2$-degree \'etale covering $\tau\colon C_{\rm{dou}}^{0}\to (C_{\rm{dou}}^{0})^{\vee}$
\end{lem}
\begin{proof} If $p$ is a simple Gauss double point, it comes together with a unique other point ${\hat{p}} \in T_pX$ where $X_p$ has the other node exactly on $\hat p$ and $T_{\hat p}X=T_pX$; that is 
$\phi_{\rm{Gauss}}(p)=\phi_{\rm{Gauss}}({\hat{p}})$. By definition of the double cover curve there are no other points of $C_{\rm{dou}}^{0}$ over $\phi_{\rm{Gauss}}(p)$.
\end{proof}

We define $\Sigma_{\rm{dou}}^{0}\subset\mP^3$ the open surface swept by the lines $l_{p,\hat p}:=\langle p,\hat p \rangle\subset\mP^3$ where $p\in C_{\rm{dou}}^{0}$. We notice that: $1)$ $[l_{p,\hat p}]\in S$ and $2)$ if $p_1\neq p$ then 
$l_{p_1, {\hat{p_1}}}\cap l_{p,\hat p}=\emptyset$. This implies that we can define an injective morphism $ (j_{\rm{dou}}^{0})^{\vee} \colon  (C_{\rm{dou}}^{0})^{\vee}  \hookrightarrow S$ defined by $ (j_{\rm{dou}}^{0})^{\vee}\colon \phi_{\rm{Gauss}}(p)\mapsto [l_{p,\hat p}]$. By the normalisation morphism $\nu_{{\rm{dou}}}^{\vee}\colon {\widetilde{ C^{\vee}_{{\rm{dou}}}} }\to C_{{\rm{dou}}}^{\vee}$ we obtain a global morphism generically of degree 1:
$$j^{\vee}_{{\rm{dou}}}\colon {\widetilde{ C^{\vee}_{{\rm{dou}}}} }\rightarrow S.
$$
We set $$C:=j^{\vee}_{{\rm{dou}}}( {\widetilde{ C^{\vee}_{{\rm{dou}}}} }  ).$$
In Definition \ref{rigariga} we have called $\Sigma_{\rm{dou} }\hookrightarrow\mP^3$ the almost ruled surface of  Gauss double points.
\begin{lem}\label{chiusura} If ${\overline{\Sigma_{\rm{dou}}^{0}}}\subset\mP^3$ is the projective closure of $\Sigma_{\rm{dou}}^{0}$ then it holds that $\Sigma_{\rm{dou}}={\overline{\Sigma_{\rm{dou}}^{0}}}\subset\mP^3$. Moreover $C_{\rm{dou}}\hookrightarrow \Sigma_{\rm{dou}}$.
\end{lem}
\begin{proof} By Proposition \ref{propietanumeriche} $C_{\rm{dou}}$ is irreducible, then the first claim follows by the universal property of $\mathbb G$. By construction it holds that $C_{\rm{dou}}^{0}\hookrightarrow \Sigma_{\rm{dou}}$ then  $C_{\rm{dou}}\hookrightarrow \Sigma_{\rm{dou}}$ since $C_{\rm{dou}}$ is the closure inside $\mP^3$ of $C_{\rm{dou}}^{0}$.
\end{proof}
\noindent Let $\pi_C\colon \Sigma\to C$ be the ruled surface obtained by the pull-back of $\pi_S\colon\mP(\sQ_S)\to S$ via the natural inclusion $C\hookrightarrow S$.
\begin{lem}\label{sigmairriducibile} The curve $C$ is irreducible and the surface $\Sigma$ is irreducible.
\end{lem}
\begin{proof} By Proposition \ref{propietanumeriche} and by its construction the curve $C$ is irreducible. Since $\pi_S\colon\mP(\sQ_S)\to S$ is a fiber bundle then the claim follows.
\end{proof}

\noindent 
Consider again the surface $\Sigma_{\rm{dou} }\hookrightarrow\mP^3$ of  Gauss double points. Note that the pull-back $\Sigma^{0}\to (C_{\rm{dou}}^{0})^{\vee}$ of $\pi_C\colon \Sigma\to C$ via the inclusion $(j_{\rm{dou}}^{0})^{\vee}\colon (C_{\rm{dou}}^{0})^ {\vee}\hookrightarrow S$ is a smooth open ruled surface. Now we obviously have:
$$\Sigma_{\rm{dou}}= f(\Sigma)$$

\begin{cor} There exists a rational map $\mu\colon \Sigma_{\rm{dou} }\dashrightarrow {\widetilde{ C^{\vee}_{{\rm{dou}}}} }$ which induces the natural structure of smooth ruled surface on 
$\Sigma^{0}\to ( C_{\rm{dou}}^{0}) ^{\vee}$
\end{cor}
\begin{proof} The pull-back $\pi'\colon{\widetilde{\Sigma}} \to {\widetilde{ C^{\vee}_{{\rm{dou}}}} }$ of $\pi_S\colon\mP(\sQ_S)\to S$
via the morphism $j^{\vee}_{{\rm{dou}}}\colon {\widetilde{ C^{\vee}_{{\rm{dou}}}} }\rightarrow S$ is a smooth surface which is mapped birationally onto $\Sigma_{\rm{dou}}$ in a way compatible with the morphism $\phi_{\rm{Gauss}|C_{\rm{dou}}}\colon C_{\rm{dou}}\to C^{\vee}_{\rm{dou}}$.
\end{proof}

\begin{lem}\label{osculanza} The almost ruled surface of  Gauss double points is irreducible and it osculates $X$ along $C_{\rm{dou}}$; that is the restriction $\Sigma_{{\rm{dou}}|X}$ is a subscheme of $X$ which contains $2C_{\rm{dou}}$.
\end{lem}
\begin{proof} By construction we have seen that $\Sigma_{\rm{dou}}= f(\Sigma)$. Then the first claim follows by Lemma \ref{sigmairriducibile}.
We show that the open ruled surface $\Sigma^{0}$ osculates $X$ along $C_{\rm{dou}}^{0}$; but this follows by definition since any fiber $l$ of $\Sigma^{0}\to (C_{\rm{dou}}^{0} )^{\vee}$ osculates $X$ along the corresponding two simple Gauss points $p,\hat p$ such that $X_{|l}=2p+2{ \hat p}\in{\rm{Div}} (l)$.
\end{proof}

\subsection{The geometry of the morphism of bitangents}
Now we start the study of the morphism $f\colon\mP(\sQ_S)\to\mP^3$.

\begin{lem}\label{finitezza} The morphism $f\colon\mP(\sQ_S)\to\mP^3$ is finite of degree $12$
\end{lem}
\begin{proof} Let $p\in\mP^3$ such that the $f$-fiber is of positive dimension. This means that there are infinite bitangent lines through $p$. Then the polar cubic ${\rm{Pol}}_{p}(X)$ has at least a component swept by lines through $p$. The restriction of ${\rm{Pol}}_{p}(X)$ to $X$ is non reduced. This implies that this restriction is a divisor of type $2D+A$. Since we are assuming that ${\rm{Pic}}(X)=[H_{ \mP^{3}_{\mid X} }] \mathbb Z$, $X$ does not contain curves of degree $\leq 3$. Then the only possibility is that $D$ is an hyperplane section. This  implies that $S$ contains the rational curve which parameterises the pencil of lines contained in a plane and passing through $p$. This is a contradiction. Indeed there exists a unamified covering $S_X\to S$ where $S_X$ is an irregular surface; the details of the proof are in \cite[Proposition 2.4]{T}and in \cite[Proposition 3.1]{T}. 
Another self-contained proof is in \cite[Lemma 1.1]{W}. See also \S 3 of \cite{CZ}. Finally by \cite[Corollary A. 3, p. 53]{W} it is known that the Albanese morphism  ${\rm{alb}}\colon S_{X}\to {\rm{Alb}}(S_X)$ is a closed immersion. In particular there are no rational curves on $S$.
\end{proof}
\subsubsection{The ramification divisor}
\begin{lem}\label{contattoX} The surface of contact points is a subdivisor of $R(f)$of multiplicity $1$. Moreover the induced morphism $\rho\colon Y\to X$ is of degree $6$ and $f^{\star}\sO_{\mP^3}(X)=\sO_{\mP(\sQ_S)}(2Y)$
\end{lem}
\begin{proof} By definition $f^{-1}(X)= Y$. Since any bitangent line through a point $p\in X$  is contained in $T_pX$, any bitangent line through $p$ gives a ramification point for the morphism ${\widetilde{X_p}}\rightarrow \mP^1$ given by the pencil of lines inside $T_pX$ with focal point $p$. This and Lemma \ref{finitezza} imply that the induced morphism $\rho\colon Y\to X$ is finite of degree $6$. Since $f^{-1}(X)= Y$ then Chow groups projection formula implies that $12[X] =f_{\star}f^{\star}[X]=f_{\star}a[Y]=af_{\star}[Y]=6a[X]$. Then $a=2$.
\end{proof}

\begin{lem}\label{sigma} The surface $\Sigma$ belongs to $R(f)$.
\end{lem}
\begin{proof} This follows easily by a local count. We use notation of Proposition \ref{lisciezza}. We consider  a line $l\subset \mP^3$ which is bitangent to $X$. W.l.o.g. we can assume that $l:=(x_2=x_3=0)$ and that the two points where $l$ is tangent to $X$ are $P=(1:0:0:0)$, $P_{\lambda}(1:\lambda:0:0)$. For a while assume $\lambda\neq 0$. The general point $p\in l$ is then given by a parameter $t\in\mathbb C$ and $p=(1:t:0:0)$. We easily can write the tangent space to $\mP(\sQ_S)$ at the point $([l],p)$. Indeed let $w_1=\frac{x_1}{x_0}-t$, $w_2=\frac{x_2}{x_0}$, $w_3=\frac{x_3}{x_0}$ such that $(w_1,w_2,w_3)$ is a system of local coordinates around the point $p\in\mP^3$ and let $(u_0,u_1,u_2,u_3)$ be a local system of coordinates at the point $[l]\in \mathbb G$, that is a line $r$ near to $l$ is parameterised by
\[
\begin{cases}
u_0x_0+u_1w_1+u_1t=w_2\\
u_2+u_3w_1+u_3t=w_3
\end{cases}
\]
Then (locally)  inside $\mP^3\times \mathbb G$, where we take coordinates $(w_1,w_2,w_3,u_0,u_1,u_2,u_3)$ the tangent space $T_{([l],p)}\mP(\sQ_S)$ is (locally) given by 

\[
\begin{cases}
u_0+u_1t-w_2=0\\
u_2+u_3t-w_3=0\\
u_0g(1:0)+u_2h(1:0)=0 \\
(u_0+\lambda u_1)g(1:\lambda)+(u_2+\lambda u_3)h(1:\lambda)=0
\end{cases}
\]
since the proof of Proposition \ref{lisciezza}. The morphism $f\colon\mP(\sQ_S)\to\mP^3$ is given by the restriction to $\mP(\sQ_S)$ of the natural projection $\rho\colon\mP(\sQ)=\mP(\Omega^1_{\mP^3})\to\mP^3$, which is, locally, $(w_1,w_2,w_3,u_0,u_1,u_2,u_3)\mapsto  (w_1,w_2,w_3)$. Now suppose for a while we are in the general case where $h(1:0)=H(1:0:0:0)\neq 0$ and  $h(1:\lambda)=H(1:\lambda:0:0)\neq 0$.
We take $(w_1, u_0,u_1)$ as local coordinates for the threefold $\mP(\sQ_S)$ around the point $([l],p)$ . The matrix of the differential $d_{([l],p)}f$ is then given by

\[ \begin{pmatrix} 1 & 0& 0 \\
0&1& t\\
0& -(\lambda-t)\frac{g(1:0)}{\lambda h(1:0)}-\frac{tg(1:\lambda)}{\lambda h(1:\lambda)}&-\frac{tg(1:\lambda)}{h(1:\lambda)}
\end{pmatrix} \]
Hence its determinant is zero iff
$
0=t(t-\lambda)\cdot {\rm{det}} \bigl( \begin{smallmatrix} g(1:0) & h(1:0) \\
g(1:\lambda)& h(1:\lambda)
\end{smallmatrix}
\bigr)$; we immediately see that $(0=t(t-\lambda))$ is the local equation of $Y\subset \mP(\sQ_S)$. The condition $0={\rm{det}} \bigl( \begin{smallmatrix} g(1:0) & h(1:0) \\
g(1:\lambda)& h(1:\lambda)
\end{smallmatrix}
\bigr)$ is independent of $t$. This last condition means that when it is satisfied this occurs for all the points which belong to the bitangent line $l$. Finally it is a trivial verification on the equation $$
F(x_0:x_1:x_2:x_3)=x_{1}^2(x_1-\lambda x_0)^2 +x_2G(x_0:x_1:x_2:x_3)+x_3H(x_0:x_1:x_2:x_3)
$$
of $X$ to see that generically the condition $0={\rm{det}} \bigl( \begin{smallmatrix} g(1:0) & h(1:0) \\
g(1:\lambda)& h(1:\lambda)
\end{smallmatrix}
\bigr)$ occurs iff $T_PX=T_{P_{\lambda}}X$ and $X_P=T_PX\cap X$ is a quartic with two nodes respectively on $P$ and on $P_{\lambda}$. Analogue computations hold if we are in more special cases where $h(1:\lambda)=0$ and $h(1:0)\neq 0$ or $h(1:\lambda)\neq 0$ and $h(1:0)= 0$ or $g(1:\lambda)=0$ and $g(1:0)\neq 0$ or $g(1:\lambda)\neq 0$ and $g(1:0)= 0$ or where $\lambda=0$. All these verifications are analogue to the one shown above. The condition above means that the open surface $\Sigma_0:=\pi_{S}^{-1}( (j_{\rm{dou}}^{0})^{\vee}  ( (C_{\rm{dou}}^{0})^{\vee} ))$ is inside $R(f)$. This implies the claim.
\end{proof}

\begin{lem}\label{ramif} The divisor $R(f)$ is linearly equivalent to $Y +\pi_{S}^{\star} (4H_{\mathbb G|S})$.
\end{lem}
\begin{proof} 
By Theorem \ref{formulae} $K_{S}=3H_{\mathbb G|S}+\sigma$. Since the first Chern class of $\sQ_S$ is $H_{\mathbb G|S}$ then by the standard formula of the canonical class divisor we conclude that
$K_{\mP(\sQ_S)}\sim -2R+\pi_{S}^{\star} (4H_{\mathbb G|S}+\sigma)$. By Proposition \ref{finitezza} we know that $K_{\mP(\sQ_S)} \sim  f^{\star} (K_{\mP^{3}}) +R(f)\sim f^{\star}(-X)+R(f)$ since $X\in |4H_{ \mP^{3} }|$. Then since $\sigma\sim -\sigma$ it holds that 
\begin{equation}\label{vediva}
R(f)\sim  f^{\star}(X)  -(2R+\pi_{S}^{\star} (\sigma)) +\pi_{S}^{\star} (4H_{\mathbb G|S}).
\end{equation}
By Proposition \ref{contattoX} we know that sheaf theoretically $f^{\star}\sO_{\mP^3}(X)=\sO_{\mP(\sQ_S)}(2Y)$ and by Proposition \ref{ilrivestimento doppio ramificato} we know that $Y\in |2R+\pi_S^{\star}(\sigma)|$. This implies that $f^{\star}(X)  -(2R+\pi_{S}^{\star} (\sigma))\sim Y$. Hence by substitution inside Equation (\ref{vediva}) we have
\begin{equation}\label{vedivavadue}
R(f)\sim Y +\pi_{S}^{\star} (4H_{\mathbb G|S}).
\end{equation}
\end{proof}
\subsubsection{The Branch divisor}
We now study the branch divisor of $f\colon\mP(\sQ_S)\to\mP^3$. The claim follows by a delicate computation on Chow groups. Since there is no possibility of misunderstanding we will indicate the Chow class $[A]$ of the cycle $A$ simply by $A$. We recall here that a line $l$ inside $\mP^3$ is given as $l= H_{\mP^3}\cdot H_{\mP^3}=H^2_{\mP^3}$.

\begin{prop}\label{diramazio} $B(f)=X+\Sigma_{\rm{dou}}$
\end{prop}
\begin{proof}  By projection formula
$$
l\cdot f_{\star} R(f)=H^2_{\mP^3}\cdot f_{\star} R(f)=f^{\star}(H^2_{\mP^3})\cdot R(f).
$$
By Lemma \ref{ramif} 
\begin{equation}\label{vedivava}
R(f)\sim Y +\pi_{S}^{\star} (4H_{\mathbb G|S}),
\end{equation}
hence
$$
f^{\star}(H^2_{\mP^3})\cdot R(f)=f^{\star}(H^2_{\mP^3})\cdot Y+f^{\star}(H^2_{\mP^3})\cdot \pi_{S}^{\star} (4H_{\mathbb G|S}).
$$
Since $\sigma$ is a torsion element and $R= f^{\star}(H_{\mP^3})$ it follows that $$f^{\star}(H^2_{\mP^3})\cdot Y=f^{\star}(H^2_{\mP^3})\cdot 2f^{\star}(H_{\mP^3})=2(f^{\star}(H_{\mP^3}))^{3}=24.$$
It remains to compute the natural number: $f^{\star}(H^2_{\mP^3})\cdot \pi_{S}^{\star} (4H_{\mathbb G|S})$. We can find it by divisors restriction on $Y$. Inside ${\rm{Pic}}(\mP(\sQ_S)$ the class $Y$ is numerically equivalent to $2f^{\star}(H_{\mP^3})$. Then in the Chow algebra of $\mP(\sQ_S)$ it holds that:
$$
f^{\star}(H^2_{\mP^3})\cdot \pi_{S}^{\star} (4H_{\mathbb G|S})=f^{\star}(H_{\mP^3})\cdot f^{\star}(H_{\mP^3})\cdot \pi_{S}^{\star} (4H_{\mathbb G|S})=2\cdot f^{\star}(H_{\mP^3})\cdot
Y\cdot \pi_{S}^{\star} (H_{\mathbb G|S}).
$$
We point out that the number $f^{\star}(H_{\mP^3})\cdot Y\cdot \pi_{S}^{\star} (H_{\mathbb G|S})$ coincides with the intersection number of the following two divisors on $Y$: $(f^{\star}(H_{\mP^3}))_{\mid Y}$ and  $(\pi_{S}^{\star} (H_{\mathbb G|S}))_{\mid Y}$. Now we carry on this computation on $Y$ seen as a divisor inside $\mP(\Omega^1_X(1)))$, where we can use the conversion rules recalled in Lemma \ref{relativeonX}. By Proposition \ref{sopraX}
we know that as a class inside ${\rm{Pic}}(\mP(\Omega^1_X(1)))$, $Y\in |6T+8 \rho_{X}^{\star}(h)|$. Then
$$
(f^{\star}(H_{\mP^3}))_{\mid Y}\cdot (\pi_{S}^{\star} (H_{\mathbb G|S}))_{\mid Y}=(6T+8\rho_{X}^{\star}(h))\cdot \rho_{X}^{\star}(h)\cdot (T+2\rho_{X}^{\star}(h))
$$
and since on ${\rm{Pic}}(\mP(\Omega^1_X(1)))$ it holds that $(\rho_{X}^{\star}(h))^3=0$, $T\cdot (\rho_{X}^{\star}(h))^2=4$ and $T^2\cdot (\rho_{X}^{\star}(h))=0$ it follows that
$$
(6T+8\rho_{X}^{\star}(h))\cdot \rho_{X}^{\star}(h)\cdot (T+2\rho_{X}^{\star}(h))=
6T^2\cdot \rho_{X}^{\star}(h)+ 20 T\cdot \rho_{X}^{\star}(h)^2=80.
$$
We have shown that $l\cdot f_{\star} R(f)=24+160$. On the other hand by Lemma \ref{sigma} we know that $\Sigma_{\rm{dou}}$ is a subdivisor of $B(f)$ since $f(\Sigma)=\Sigma_{\rm{dou}}$. By Lemma \ref{osculanza} we know that $\Sigma_{\rm{dou}}$ restricts on $X$ at least to $2C_{\rm{dou}}$. By Theorem \ref{dovesta} $C_{\rm{dou}}\in |80h|$. Then $\Sigma_{\rm{dou}}\in |b H_{\mP^3}|$ where $b\geq 160$. Moreover by Lemma \ref{contattoX} we certainly have that $6X+\Sigma_{\rm{dou}}$ is at least a subdivisor of $f_{\star} R(f)$, but

$$
(6X+\Sigma_{\rm{dou}})\cdot l= 6X\cdot l+bH_{\mP^3}\cdot l=24+160.
$$
This implies that $b=160$ and that $X+\Sigma_{\rm{dou}}$ is exactly the divisor $B(f)$.
\end{proof}

\subsection{Geometrical consequences on the geometry of Gauss curves}
By Proposition \ref{diramazio} and by its proof we can complete the geometrical picture behind Proposition \ref{normalizzodouble}.
\begin{lem}\label{ramif2} $R(f)=Y+\Sigma$. In particular $\Sigma\in |\pi_{S}^{\star} (4H_{\mathbb G|S})|$
\end{lem}
\begin{proof} Consider the two morphisms $f\colon \mP(\sQ_S)\to\mP^3$ and $\pi_S\colon  \mP(\sQ_S)\to S$.
The divisor $Y+\Sigma$ is a subdivisor of $R(f)$ and by Lemma \ref{ramif} $R(f)\sim Y +\pi_{S}^{\star} (4H_{\mathbb G|S})$. By construction $\Sigma\in |\pi_{S}^{\star} (C)|$. Hence $C$ is a subdivisor of a divisor
 $D\in |4H_{\mathbb G|S}|$ but looking to the $f$-direction, by the same technique used in the proof of Lemma \ref{diramazio} we have that 

$$f^{\star}(H^2_{\mP^3})\cdot \pi_{S}^{\star} (4H_{\mathbb G|S})=f^{\star}(H^2_{\mP^3})\cdot \Sigma$$
hence  $C\in |4H_{\mathbb G|S}|$, and the claim follows.
\end{proof}
\subsection{The proof of Proposition \ref{normalizzodouble}}
\label{subs.proposizionegeometrica}
By the proof of Lemma \ref{ramif2} $C\in  |4H_{\mathbb G|S}|$ and by Lemma \ref{sigmairriducibile} $C$ is irreducible. Hence by Lemma \ref{chowclass}  and by Theorem \ref{formulae} $(1)$ it holds that $C$ is an irreducible curve of arithmetical genus $\rho_{a}(C)=561$. By Proposition \ref{dualedouble} we know that ${\widetilde{ C^{\vee}_{{\rm{dou}}}} }$ has genus $561$. Then the morphism $j^{\vee}_{{\rm{dou}}}\colon {\widetilde{ C^{\vee}_{{\rm{dou}}}} }\rightarrow C\subset S$ is an embedding.

By Proposition \ref{ilrivestimento doppio ramificato} we know that $\pi\colon Y\to S$ is branched on the curve of hyperflexes which is in $|2H_{\mathbb G|S}|$. Then an analogue argument shows that $j_{{\rm{dou}}} \colon {\widetilde{ C_{{\rm{dou}}}}}\hookrightarrow Y$ too is an embedding. 

By construction we have that the diagram (\ref{diagrammadellenormalizzodouble}) is commutative. We have shown Proposition \ref{normalizzodouble}.

\subsection{The proof of the Main Theorem}

\label{subs.teoremageometrico}
We have shown in Proposition \ref{lisciezza} that $S$ is smooth. By Proposition \ref{ilrivestimento doppio ramificato} we know that $Y$ is smooth. By Proposition \ref{propietanumeriche} we have a full classification of the singularities of the double cover curve $C_{\rm{dou}}\subset X$. By Proposition \ref{dualedouble} we have a full classification of the singularities of the dual curve  $C^{\vee}_{\rm{dou}}$. By Lemma \ref{finitezza} the morphism $f\colon\mP(\sQ_S)\to\mP^3$ is finite of degree $12$. By the proof of Proposition \ref{diramazio} we know that $\Sigma_{\rm{dou}}$ is a surface of degree 160. By Lemma \ref{osculanza} and by the proof of Proposition \ref{diramazio} we know that $\Sigma_{{\rm{dou}}_{\mid X}}$ is exactly $2C_{\rm{dou}}$. Finally by proposition \ref{diramazio} $B(f)=X+\Sigma_{\rm{dou}}$, by Lemma \ref{ramif2} $R(f)=Y+\Sigma$, by Lemma \ref{contattoX} the morphism $f_{\mid Y}=\rho\colon Y\to X$ is of degree $6$ and by the proof of Lemma \ref{ramif2} and by construction the morphism $f_{\mid\Sigma}\colon \Sigma\to\Sigma_{\rm{dou}}$ is birational. We have shown the Main Theorem.


\begin{thebibliography}{1}

\bibitem{CZ}
 P. Corvaja, F. Zucconi, \emph{Bitangents to the quartic surface and infinitesimal deformations}, preprint available at https://arxiv.org/abs/1910.01365.
\bibitem{Hu}
D. Huybrechts. \emph{Lectures on K3 surfaces,}  Cambridge studies in advanced mathematics (2016).
\bibitem{Is}[Is]
V. A., Iskovskih, \emph{Fano 3-folds} Izv. Akad. Nauk SSSR Ser. Mat. 41, (1977), 516--562. English transl. in Math. USSR Izvestija. 11, (1977), 485--527.
 \bibitem{T}
 A. S. Tihomirov, \emph{The geometry of the Fano surface of the double cover of $\mP^3$ branched in a quartic}, Math. USSR Izv \text{16}, (1982), (2), 373--398.
\bibitem{W}
G.E. Welters \emph{Abel-Jacobi isogenies for certain types of Fano threefolds}
Mathematical Centre Tracts 141, printed at the Mathematical Centre, 413 Kruislan, Amsterdam 1981.
\bibitem{Wi}
J. Witaszek, \emph{The Geometry of Smooth Quartics} Master`s Thesis Mathematics Mathematisch-Naturwissenschaftliche Fakultt\"at der Rheinischen Friederich-Wilhelms-Universit\"at Bonn (2014)

\end{thebibliography}
\end{document}